\documentclass[11pt, reqno]{amsart}
\usepackage[utf8]{inputenc}
\usepackage{mathrsfs}
\usepackage{mathtools}
\usepackage{amsmath}
\usepackage{amsthm}
\numberwithin{equation}{section}

\usepackage[left=2.5cm, right=2.5cm]{geometry}

\usepackage{multicol}
\usepackage{tikz,pgf}
\usepackage{tikz-cd}
\usepackage{float}
\usepackage[all,knot,arc,color,web]{xy}
\usepackage{graphicx}
\xyoption{curve}
\usepackage{bbold}
\usepackage{fancyhdr}
\theoremstyle{plain}
\newtheorem{theorem}{Theorem}[section]
\newtheorem{proposition}[theorem]{Proposition}
\newtheorem{corollary}[theorem]{Corollary}
\newtheorem{lemma}[theorem]{Lemma}

\theoremstyle{definition}
\long\def\comment#1{}

\title{Bialy-Mironov type rigidity for centrally symmetric symplectic billiards}
\author{Luca Baracco, Olga Bernardi, Alessandra Nardi}
\address{Dipartimento di Matematica Tullio Levi-Civita, Universit\`a di Padova, via Trieste 63, 35121 Padova, Italy}
\email{baracco@math.unipd.it, obern@math.unipd.it, nardi@math.unipd.it}
\date{}

\begin{document}
\maketitle
\begin{abstract}
The aim of the present paper is to establish a Bialy-Mironov type rigidity for centrally symmetric symplectic billiards. For a centrally symmetric $C^2$ strongly-convex domain $D$ with boundary $\partial D$, assume that the symplectic billiard map has a (simple) continuous invariant curve $\delta \subset \mathcal{P}$ of rotation number $1/4$ (winding once around $\partial D$) and consisting only of $4$-periodic orbits. If one of the parts between $\delta$ and each boundary of the phase-space is entirely foliated by continuous invariant closed (not null-homotopic) curves, then $\partial D$ is an ellipse. The differences with Birkhoff billiards are essentially two: it is possible to assume the existence of the foliation in one of the parts of the phase-space detected by the curve $\delta$, and the result is obtained by tracing back the problem directly to the totally integrable case. 
\end{abstract}
\section{Introduction}
Symplectic billiards were introduced by P. Albers and S. Tabachnikov \cite{ALB} in 2017, as a natural variation of Birkhoff billiards with the inner area --instead of the length-- as generating function. Let $D$ be a convex planar domain with boundary $\partial D$. Among the convex $n$-gons inscribed in $D$, let $P$ be the one of greatest area. It is well-known (see e.g. \cite[Page 9]{Toth}) that such a $P$ has the (necessarily) property that the tangent line at each vertex is parallel to the diagonal passing through the neighboring vertices. As a consequence, the dynamical formulation of the symplectic billiard map is the following: three points $x,y,z$ on $\partial D$ are three consecutive points of a symplectic billiard orbit if and only if the tangent at the second point $y$ is parallel to the line connecting $x$ and $z$. We refer to Figure \ref{mappa}. 
\begin{figure}
\centering
\includegraphics[scale=0.9]{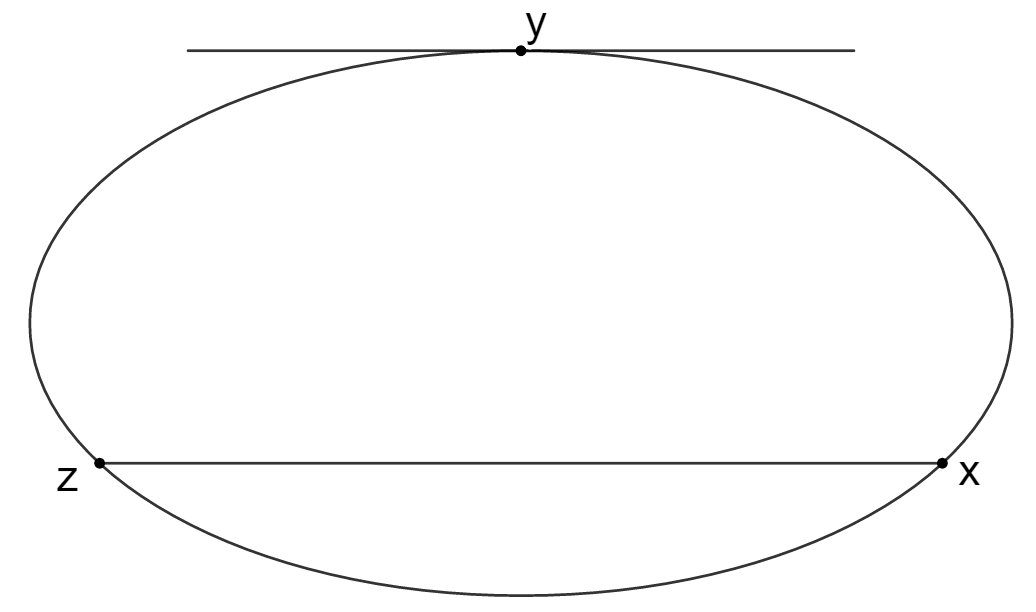}
\caption{The symplectic billiard map.}
\label{mappa}
\end{figure}
As proved in \cite[Section 2]{ALB}, see also Section \ref{DS} of the present paper, the symplectic billiard map turns out to be a twist map, preserving an area form; moreover, it commutes with affine transformations of the plane. \\
\indent Crucial questions for any billiard dynamics comprehend integrability, which means the existence of a regular (i.e. at least $C^0$) foliation of the phase-space consisting of invariant, not null-homotopic curves. There are different ways to define integrability, including local and/or perturbative notions. We refer to \cite{Alf} for an exhaustive discussion in the case of Birkhoff billiards. In particular, a billiard is called {\it totally --or full globally-- integrable} if the foliation fill the whole phase-space. The first celebrated result that {\it totally integrable Birkhoff billiards are circles} was proved by M. Bialy \cite{bi} in 1993. It is based on a generalization of Hopf's argument, constructing a (discrete) Jacobi field without conjugated points, and uses the usual planar isoperimetric inequality. We refer to \cite{Wo} for an alternative subsequent proof of this theorem based on genuine dynamical arguments. Finally, we notice that an additional quicker proof can also be performed by using different recent coordinates introduced in \cite[Section 3.1]{BIAMI}, see also \cite[Formula (14)]{BG}. In 2023, M. Bialy \cite{BO} establishes --with the same integral-geometric approach-- the corresponding rigidity phenomenon for outer billiards: {\it if an outer billiard is totally integrable, then the boundary of the billiard table is an ellipse}. It is worth noting that, in such a case, a specific generating function as well as the use of suitable weights --in order to overcome the non-compactness of the phase-space-- play a fundamental role. Finally, the result is reduced to the Blaschke-Santalo inequality. In the same year, the first two authors of the present paper proved \cite{BaBe} that {\it the only totally integrable symplectic billiards are ellipses} by using --beyond the well-consolidated integral approach coming from Hopf's method-- the affine equivariance of the symplectic billiard map in order to take back the problem to the isoperimetric inequality. \\
\indent It is consequently quite natural trying to apply the previous successful framework to search for other rigidity results, by relaxing the totally integrability assumption. In such a direction, a fundamental contribution is due to M. Bialy and A.E. Mironov \cite{BM} in 2022, proving the so-called Birkhoff-Poritsky conjecture for centrally symmetric $C^2$ strongly-convex (i.e. with positive curvature) domains $D$. Their main theorem can be stated as follows. {\it Assume that the Birkhoff billiard map of $\partial D$ has a (simple) continuous invariant curve of rotation number $1/4$ (winding once around the phase-space cylinder) and consisting only of $4$-periodic orbits. Moreover, suppose that the domain between this invariant curve and the boundary of the phase-space cylinder corresponding to the identity is entirely foliated by continuous invariant closed (not null-homotopic) curves. Then $\partial D$ is an ellipse}. As clearly explained by the authors, the main ingredients of the proof are the use of the non-standard generating function for convex Birkhoff billiards introduced in \cite{BIAMI}, the accurate study of the properties of the above $4$-periodic orbits and --in order to conclude-- the Hopf's approach combined with Wirtinger's inequality. We also recall the contemporary paper \cite{BT} by M. Bialy and D. Tsodikovich where billiard tables with rotational symmetries of order greater than $3$ (therefore  strengthening the centrally symmetric hypothesis) are considered. In this article, the geometric assumptions allow the authors to weaken the total integrability to a region of the phase-space (in particular, to a neighborhood of a specific invariant curve formed by periodic points whose rotation number is linked to the order of symmetry of the billiard table). The results, essentially based on the remarkable structure of the invariant curves as well as the billiard tables, apply to Birkhoff, outer, symplectic and Minkowski billiards. \\ 
\indent The aim of the present paper is to establish a Bialy-Mironov type rigidity for centrally symmetric symplectic billiards, as stated in the next theorem. In the sequel, we indicate $T: \mathcal{P} \to \mathcal{P}$ the symplectic billiard map and we refer to the beginning of Section \ref{DS} for details on the precise definition of the phase-space $\mathcal{P}$.  

\begin{theorem}\label{4-per} Let $D$ be a centrally symmetric $C^2$ strongly-convex domain with boundary $\partial D$. Assume that the symplectic billiard map $T: \mathcal{P} \to \mathcal{P}$ of $\partial D$ has a (simple) continuous invariant curve $\delta \subset \mathcal{P}$ of rotation number $1/4$ (winding once around $\partial D$) and consisting only of $4$-periodic orbits. If one of the parts between $\delta$ and each boundary of the phase-space $\mathcal{P}$ is entirely foliated by continuous invariant closed (not null-homotopic) curves, then $\partial D$ is an ellipse. 
\end{theorem}

\noindent We stress that, differently from Birkhoff case, the symmetric properties of the generating function for symplectic billiards in centrally symmetric tables allow to assume the existence of the foliation of invariant curves in one of the parts of the phase-space detected by the curve $\delta$. Moreover, another difference from the ``classical case'' is that the proof of the above theorem is based on some technical lemmas (similar to Lemmas of \cite[Section 4]{BaBe}) which allow us to trace back the problem to the totally integrable case studied in \cite{BaBe}. On the contrary, the Birkhoff case needs ad hoc arguments as explained right above. As a consequence, the techniques used to prove Theorem \ref{4-per} don't fit if we consider smaller regions, e.g. by replacing $\delta$ with an invariant curve of rotation number $1/2^n$ with $n \ge 3$ or if we investigate on other billiard dynamics. We refer for example to the outer billiard case, which is particular both for the non-compactness of the phase-space and for the expression of the generating function. \\
\indent Straightforward consequences of Theorem \ref{4-per} are the following corollaries (corresponding respectively to Corollaries 1.3 and 1.2 in \cite{BM}). 

\begin{corollary}
    If the symplectic billiard map $T:\mathcal{P}\to\mathcal{P}$ has a $C^1$ first integral with non vanishing gradient on one of the closed parts between $\delta$ and each boundary of the phase-space $\mathcal{P}$, then $\partial D$ is an ellipse.
\end{corollary}

\noindent In fact, a $C^1$  first integral with non vanishing gradient in a region of the phase-space induces a foliation by continuous invariant closed (not null-homotopic) curves. Since a boundary of this region is, by hypothesis, the curve $\delta$, Theorem \ref{4-per} immediately applies. 
\begin{corollary} \label{1 e 2}
Suppose that one of the next two hypotheses holds.
\begin{itemize}
\item[$1.$] A neighborhood of $\partial D$ is $C^1$-foliated by convex caustics of rotation numbers $(0,1/4]$;
\item[$2.$] A neighborhood of the center of symmetry of $D$ is $C^1$-foliated by convex caustics of rotation numbers $[1/4,1/2)$.
\end{itemize}
Then $\partial D$ is an ellipse.
\end{corollary}
\noindent The above corollary follows from these facts. By a standard argument, a foliation of differentiable convex caustics of given rotation numbers corresponds to a foliation of continuous invariant closed (not null-homotopic) curves with the same rotation numbers (see e.g. \cite[Pages 44-45]{Si}). In order to conclude the existence of the curve $\delta$, consisting only of periodic orbits, we need the foliation to be $C^1$. In fact, in such a case, the area form preserved by $T$ induces an absolute continuous invariant measure on each leaf of the foliation. The invariance of such a measure assures that every invariant curve of rational rotation number is made up only of periodic orbits (non periodic orbits of rational rotation number should be asymptotic to periodic ones, \cite[Propositions 11.1.4. and 11.2.2.]{Ka}). In particular, hypothesis $1.$ or $2.$ of Corollary \ref{1 e 2} implies the ones of Theorem \ref{4-per}. \\
\indent The sequel is organized as follows. In Section \ref{DS} we present symplectic billiards and their main properties. In Section \ref{radon} we discuss the consequences on the geometry of $\partial D$ given by the existence of the curve $\delta$ and we recall the integral-geometric inequality to the base of the proof of Theorem \ref{4-per}. Section \ref{TEC} --corresponding to \cite[Section 4]{BaBe}-- is devoted to technical facts on integrals involving the area form. Finally, Section \ref{fine} concludes the proof of Theorem \ref{4-per}. \\
~\newline
\textbf{Acknowledgments.} We are thankful to Misha Bialy and Alfonso Sorrentino for the interesting and useful discussions on mathematical billiards.

\section{The dynamical system} \label{DS}
Let $D$ be a $C^2$ strictly-convex planar domain with boundary $\partial D$. Fixed the positive counter-clockwise orientation, let $\gamma:\mathbb{S}^1 \to\partial D$ be a smooth parametrization of $\partial D$ where $\mathbb{S}^1 := \mathbb{R}/2\pi \mathbb{Z}$. Since $D$ is strictly-convex, for every point $\gamma(t) \in \partial D$, there exists a unique point $\gamma(t^*) \in \partial D$ such that 
$$\omega(\gamma'(t),\gamma'(t^*)):= \det (\gamma'(t),\gamma'(t^*)) =0,$$ 
where --in the above formula-- $\gamma'(t)$ and $\gamma'(t^*)$ denote the tangent vectors at $\partial D$ in $\gamma(t)$ and $\gamma(t^*)$ respectively. \\
\indent Taking into account the orientation of $\partial D$, we define the (open, positive) phase-space $\mathcal{P}$ and the symplectic billiard dynamics $T$ as follows. 
$$\mathcal{P} := \{ (\gamma(t), \gamma(s)) \in \partial D \times \partial D: \ \gamma(t) < \gamma(s) < \gamma(t^*)\}$$ 
and 
$$T: \mathcal{P} \rightarrow \mathcal{P}, \qquad (\gamma(t_1),\gamma(t_2)) \mapsto (\gamma(t_2),\gamma (t_3))$$
where $\gamma(t_3)$ is the unique point satisfying 
$$\omega( \gamma^{\prime} (t_2),\gamma(t_3)-\gamma(t_1))=0.$$
Here below we collect some properties of the symplectic billiard map (see \cite[Section 2.1]{ALB} for all details).
\begin{itemize}
\item[$1.$] $T$ is continuous and can be continuosly extended to $\bar{\mathcal{P}}$ so that 
$$T(\gamma(t),\gamma(t)) = (\gamma(t),\gamma(t)) \quad \text{and} \quad T(\gamma(t),\gamma(t^*)) = (\gamma(t^*),\gamma(t)).$$ 
In particular, the extension $T(\gamma(t),\gamma(t^*)) = (\gamma(t^*),\gamma(t))$ is characterized by $2$-periodicity (see \cite[Lemma 2.3]{ALB}). In the sequel, we denote by
$$\gamma :=\{(\gamma(t),\gamma(t)): \ t \in \mathbb{S}^1\} \quad \text{and} \quad \gamma^* :=\{(\gamma(t),\gamma(t^*)): \ t \in \mathbb{S}^1\}$$
the two boundaries of the phase-space $\mathcal{P}$. 
\item[$2.$] As proved in \cite[Lemma 2.5]{ALB}, the area form 
$$\omega: \mathcal{P} \to \mathbb{R}, \qquad (t_1,t_2) \mapsto \omega(\gamma(t_1),\gamma(t_2))$$
is a generating function for $T$. 
\item[$3.$] According to \cite[Lemma 2.7]{ALB}, the symplectic billiard map is a twist map preserving the $2$-form 
\begin{equation} \label{2 forma}
\Omega := \omega(\gamma'(t_1),\gamma'(t_2)) dt_1 dt_2,
\end{equation}
that is $T^*\Omega = \Omega$.  
\item[$4.$] The symplectic billiard map commutes with affine transformations of the plane, since they preserve the tangent directions. 
\end{itemize}
\section{Invariant curve of $4$-periodic orbits} \label{radon}
From now on, $D$ is a centrally symmetric $C^2$ strongly-convex (i.e. with curvature $>0$) domain and we assume the center of symmetry of $D$ 
to be the origin. In such a case, beyond general properties $1.$ -- $4.$ recalled in Section \ref{DS}, we also have
\begin{itemize}
\item[$5.$] $\omega(\gamma(t_1 + \pi),\gamma(t_2)) = - \omega(\gamma(t_1),\gamma(t_2))$ and $\omega(\gamma(t_1),\gamma(t_2 + \pi)) = - \omega(\gamma(t_1),\gamma(t_2))$.
\end{itemize}
The fact that --in the centrally symmetric case-- the generating function changes sign by adding $\pi$ (half period) to one of the variables is a special property of symplectic billiards. In fact, in other billiards (e.g. Birkhoff, outer, fourth billiards, see \cite{ALB1} for details), the corresponding generating function depends also on the shape of the boundary between two points. \\
\indent The next proposition is essentially contained in \cite[Section 2.4.2.]{ALB}.
\begin{proposition} \label{parallelogramma}
Let $D$ be a centrally symmetric billiard table. Assume that the symplectic billiard map $T: \mathcal{P} \to \mathcal{P}$ of $\partial D$ has a (simple) continuous invariant curve $\delta \subset \mathcal{P}$ of rotation number $1/4$ (winding once around $\partial D$) and consisting only of $4$-periodic orbits. Then each quadrilateral corresponding to the invariant curve $\delta$ is a parallelogram. In particular, $\partial D$ is a Radon curve.
\end{proposition}
\begin{proof}
Let consider a $4$-periodic orbit in $\delta$, $\{ \gamma(t_i)\}_{i = 1}^4$. As a consequence of the symplectic billiard dynamics, we have
$$\omega(\gamma'(t_1),\gamma'(t_3)) = 0 \quad \text{and} \quad \omega(\gamma'(t_2),\gamma'(t_4)) = 0.$$
By the centrally symmetric hypothesis, it follows
$$\gamma(t_3) = \gamma(t_1 + \pi) \quad \text{and} \quad \gamma(t_4) = \gamma(t_2 + \pi).$$
This means that the (inscribed) quadrilateral with vertices in $\{ \gamma(t_i)\}_{i = 1}^4$ is a parallelogram whose diagonals intersect in the origin (center of symmetry). Notice that the area of this parallelogram is half the action along $\{ \gamma(t_i)\}_{i = 1}^4$, which is constant for the invariant curve $\delta$. As a consequence, every $4$-periodic orbit for the symplectic billiard dynamics corresponds to an inscribed parallelogram having maximal area passing through a given vertex. In particular, $\partial D$ is a Radon curve (according to \cite[Beginning of Section 3]{MS}).
\end{proof}
\noindent The simplest example of Radon curve is the ellipse. However, there are recent constructions of Radon curves based on methods both from Plane Minkowski Geometry \cite{Ra1} and from Convex Geometry \cite{BMT}. \\
\indent Let $\phi: \mathbb{S}^1 \to \mathbb{S}^1$ be the function induced by $\delta$ and $\Phi$ the lift of $\phi$ to $\mathbb{R}$. Then, as a consequence of Proposition \ref{parallelogramma}:
\begin{equation}\label{iter}
    \Phi^2(t)=t+\pi, \quad \Phi^3(t)=\Phi(t)+\pi \quad \text{and} \quad \Phi^4(t) = t + 2\pi.
    \end{equation}
In particular, $\Phi(t + \pi) = \Phi^3(t) = \Phi(t) + \pi$. Moreover, let $\Omega_{\gamma \delta}$ be the region in $[0,2\pi]^2$ corresponding to the part of the phase-space between $\gamma$ and the curve $\delta$:
$$\Omega_{\gamma \delta} := \{(t,s): \ t \in [0,2\pi] \text{ and } t \le s \le \Phi(t)\}.$$
Similarly, 
$$\Omega_{\delta\gamma^*} := \{(t,s): \ t \in [0,2\pi] \text{ and } \Phi(t) \le s \le t + \pi\}.$$
\indent From now on, $L(t_1,t_2) := \omega(\gamma(t_1),\gamma(t_2))$ denotes the generating function for $\Phi$ and $L_{ij}$ (for $i,j=1,2$) are the usual partial derivatives:
$$L_{11}(t_1,t_2) = \omega(\gamma''(t_1),\gamma(t_2)), \quad L_{22}(t_1,t_2) = \omega(\gamma(t_1),\gamma''(t_2)) \quad \text{and} \quad  L_{12}(t_1,t_2) = \omega(\gamma'(t_1),\gamma'(t_2)).$$
Therefore,
\begin{equation*} \label{estimate}
|L_{11}|, |L_{22}|, L_{12} \le K
\end{equation*}
where $K$ is a constant depending on the $C^2$-norm of $\gamma$. \\
The core of the next proposition is a nowadays well-consolidated argument by Bialy, based on the construction of a (discrete) Jacobi field without conjugated points. In the inequality's proof, the twist condition as well as the invariance of $\Omega$ with respect to $T$ play a crucial role. We refer to \cite[Section 3]{bi} for a detailed proof, also sketched in \cite[Section 3]{BaBe}. 
\begin{proposition} \label{necessaria}
Let $D$ be a $C^2$ strongly-convex billiard table. Suppose that the symplectic billiard map $T: \mathcal{P} \to \mathcal{P}$ of $\partial D$ has two (simple) continuous invariant closed (not null-homotopic) curves $\alpha$ and $\beta$, $\alpha < \beta$. Let $\Omega_{\alpha\beta}$ be the region in $[0,2\pi]^2$ corresponding to the part of the phase-space between $\alpha$ and $\beta$, which we assume to be foliated by continuous invariant closed (not null-homotopic)
curves. Then the following inequality holds:
    \begin{equation}\label{dis}
       \iint_{\Omega_{\alpha\beta}} [L_{11}(t_1,t_2)+2L_{12}(t_1,t_2)+L_{22}(t_1,t_2)] \, d\Omega\leq0.
   \end{equation}
\end{proposition}
\noindent The next section is devoted to study (\ref{dis}), in the special case where $D$ is a centrally symmetric billiard table and $\Omega_{\alpha\beta}$ is $\Omega_{\gamma \delta}$ or $\Omega_{\delta \gamma^*}$, that is (by (\ref{2 forma})):
\begin{equation} \label{eccola}
\int_{\Omega_{\gamma \delta}} [L_{11}(t_1,t_2) + 2L_{12}(t_1,t_2) + L_{22}(t_1,t_2)] L_{12}(t_1,t_2) \, dt_1 dt_2 \le 0
\end{equation}
and 
\begin{equation} \label{eccola1}
\int_{\Omega_{\delta \gamma^*}} [L_{11}(t_1,t_2) + 2L_{12}(t_1,t_2) + L_{22}(t_1,t_2)] L_{12}(t_1,t_2) \, dt_1 dt_2 \le 0.
\end{equation}

\section{Three technical lemmas} \label{TEC}
In this section, we prove three properties of integrals (\ref{eccola}) and (\ref{eccola1}), which will be useful in the sequel.  
\begin{lemma} \label{primo lemma}
For $L(t_1,t_2)=\omega(\gamma(t_1),\gamma(t_2))$ the next equalities hold:
    \begin{equation*}
        \begin{split}
            &\iint_{\Omega_{\gamma \delta}} [L_{11}(t_1,t_2)+2L_{12}(t_1,t_2)+L_{22}(t_1,t_2)]L_{12}(t_1,t_2) \, dt_1dt_2 \\
            &= \iint_{\Omega_{\delta\gamma^*}} [L_{11}(t_1,t_2)+2L_{12}(t_1,t_2)+L_{22}(t_1,t_2)]L_{12}(t_1,t_2) \, dt_1dt_2 \\
            &= \iint_{[0,\pi]^2}[L_{11}(t_1,t_2)+2L_{12}(t_1,t_2)+L_{22}(t_1,t_2)]L_{12}(t_1,t_2) \, dt_1dt_2.
        \end{split}
    \end{equation*}
\end{lemma}
\noindent \begin{proof}
\begin{figure}
\centering
\includegraphics[scale=0.65]{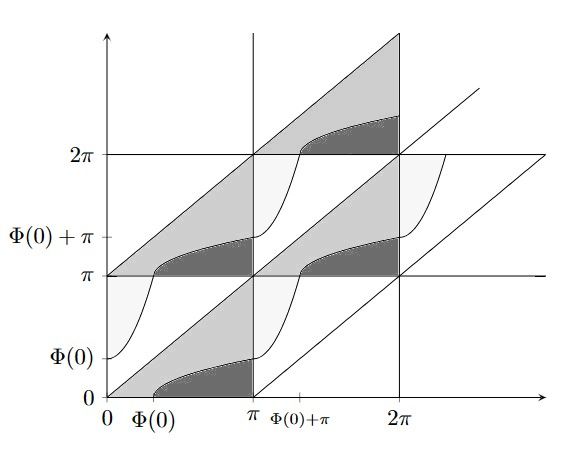}
\caption{The same colored regions have equal area.}
\label{Areas}
\end{figure}
From \eqref{iter}, we conclude that $(i)$ the graph of $\Phi$ in $[\Phi(0),\pi]$ is precisely the inverse of the graph of $\Phi$ in $[0,\Phi(0)]$, composed with the $\pi$-translation in the (positive) vertical direction; $(ii)$ the graph of $\Phi$ in $[\pi,2\pi]$ is the $\pi$-translation in the (positive) vertical direction of the graph of $\Phi$ in $[0,\pi]$. Now, as in the proof of \cite[Lemma 4.1]{BaBe}, let
$$\overline{\Omega}_{\gamma\delta}=\lbrace (t_2,t_1)\in[0,2\pi]^2:(t_1,t_2)\in\Omega_{\gamma\delta}\rbrace$$
so that --referring to Figure \ref{Areas}-- the same colored regions have equal area. Let
$$F(t_1,t_2) := [L_{11}(t_1,t_2)+2L_{12}(t_1,t_2)+L_{22}(t_1,t_2)]L_{12}(t_1,t_2).$$
Since $L_{11}(t_1,t_2)=-L_{22}(t_2,t_1)$ and $L_{12}(t_1,t_2)=-L_{12}(t_2,t_1)$, we have that 
$$\iint_{\Omega_{\gamma\delta}}  F(t_1,t_2) \, dt_1dt_2 = \iint_{\Omega_{\gamma\delta}}  F(t_2,t_1) \,dt_1dt_2= \iint_{\overline{\Omega}_{\gamma\delta}}  F(t_1,t_2) \, dt_1dt_2.$$
As a consequence, and by Property $5.$ recalled in Section \ref{radon}, we get
$$2\iint_{\Omega_{\gamma\delta}} F(t_1,t_2) \, dt_1dt_2 = \iint_{\Omega_{\gamma\delta}} F(t_1,t_2) \, dt_1dt_2 + \iint_{\overline{\Omega}_{\gamma\delta}} F(t_1,t_2) \, dt_1dt_2 = 
2 \iint_{[0,\pi]^2} F(t_1,t_2) \, dt_1dt_2$$
and the statement immediately follows.
\end{proof}
\begin{lemma} \label{secondo lemma}
    For $L(t_1,t_2)=\omega(\gamma(t_1),\gamma(t_2))$ the next equality holds:
    \begin{equation*}
        \iint_{[0,\pi]^2}[L_{11}(t_1,t_2)+L_{22}(t_1,t_2)]L_{12}(t_1,t_2)dt_1dt_2=-A(D)\int_0^\pi \omega(\gamma''(t),\gamma'(t))dt,
    \end{equation*}
where $A(D)$ denotes the area of the billiard table $D$. 
\end{lemma}
\begin{proof} The proof is similar to the one of \cite[Lemma 4.2]{BaBe}; however, in order to integrate in $[0,\pi]^2$ instead of in $[0,2\pi]^2$, the centrally symmetric hypothesis plays a crucial role. \\
Integrating by parts, we have
\[
\begin{split}
&\iint_{[0,\pi]^2}L_{11}(t_1,t_2)L_{12}(t_1,t_2) \, dt_1dt_2 =\iint_{[0,\pi]^2} \omega(\gamma''(t_1),\gamma(t_2)) \omega(\gamma'(t_1), \gamma'(t_2)) \, dt_1 dt_2 \\
&=\iint_{[0,\pi]^2} \left( \gamma_1''(t_1) \gamma_2(t_2) - \gamma_2''(t_1) \gamma_1(t_2) \right) \left( \gamma'_1(t_1) \gamma'_2(t_2)-\gamma'_2(t_1) \gamma'_1(t_2) \right) \, dt_1 dt_2
\end{split}
\]
The terms 
$$\iint_{[0,\pi]^2}\gamma_1''(t_1)\gamma_2(t_2)\gamma_2'(t_1)\gamma_1'(t_2) \, dt_1dt_2 = \iint_{[0,\pi]^2}\gamma_2''(t_1)\gamma_1(t_2)\gamma_1'(t_1)\gamma_2'(t_2) \, dt_1dt_2 = 0$$
since, by the centrally symmetric hypothesis, $\gamma'(0)=-\gamma'(\pi)$. As a consequence, we get
$$\iint_{[0,\pi]^2}L_{11}(t_1,t_2)L_{12}(t_1,t_2) \, dt_1dt_2=\iint_{[0,\pi]^2}\gamma_1''(t_1)\gamma_2'(t_1)\gamma_2'(t_2)\gamma_1(t_2)-\gamma_2''(t_1)\gamma_1'(t_1)\gamma_2'(t_2)\gamma_1(t_2) \, dt_1dt_2$$    
$$=\int_0^\pi\omega(\gamma''(t),\gamma'(t)) \, dt \int_0^\pi\gamma_1(t)\gamma_2'(t) \, dt=-\frac{A(D)}{2}\int_0^\pi\omega(\gamma''(t),\gamma'(t)) \,dt.$$
Since the same argument holds for $\iint_{[0,\pi]^2}L_{22}(t_1,t_2)L_{12}(t_1,t_2) \, dt_1dt_2$, we conclude the proof. 
\end{proof}
We now parametrize the boundary by the direction of its tangent line and denote by $e_\alpha=(-\sin{\alpha}, \cos{\alpha})$ the unitary vector forming an angle $\alpha\in[0,2\pi]$ with respect to the vertical direction $(0,1)$. Notice that for every $\alpha$ there exists a unique point $\gamma(\alpha)$ such that $\gamma'(\alpha)=\|\gamma'(\alpha)\|e_\alpha$. Let $p:[0,2\pi]\to\mathbb{R^+}$ be the corresponding support function, that is the distance from the origin of the tangent line to $\gamma$ at $\gamma(\alpha)$. We refer to Figure \ref{palpha}. 
\begin{figure}
\centering
\includegraphics[scale=0.7]{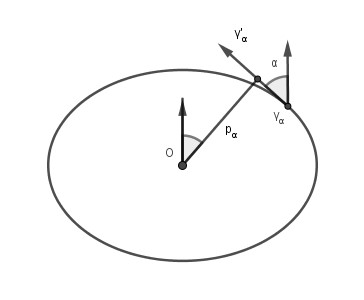}
\caption{The angle $\alpha$ and the support function $p(\alpha)$.}
\label{palpha}
\end{figure}
It is easy to see that $$\gamma(\alpha)=p'(\alpha)e_\alpha-p(\alpha)Je_\alpha$$ where $J$ is the rotation of $\frac{\pi}{2}$ in the positive verse. Consequently,
$$\gamma'(\alpha)=(p''(\alpha)+p(\alpha))e_\alpha \qquad \text{and}\qquad \gamma''(\alpha)=(p'''(\alpha)+p'(\alpha))e_\alpha+(p''(\alpha)+p(\alpha))Je_\alpha.$$ 
As already stressed in \cite[Section 4]{BaBe}, as a consequence of the strongly-convexity assumption of $\partial D$, it holds that $$(p''(\alpha) + p(\alpha))>0.$$
Therefore, the twist condition is also satisfied for $L(\alpha_1,\alpha_2) = \omega (\gamma(\alpha_1),\gamma(\alpha_2))$:
$$L_{12}(\alpha_1,\alpha_2) = (p''(\alpha_1) + p(\alpha_1))(p''(\alpha_2) + p(\alpha_2)) \sin(\alpha_2-\alpha_1) > 0$$
on $\mathcal{P} = \{(\alpha_1,\alpha_2): \ 0 < \alpha_2 - \alpha_1 < \pi \}$. \\
From the centrally symmetric assumption, it immediately follows that the functions $p$, $p'$ and $p''$ are $\pi$-periodic. Moreover, with such a parametrization, the result of Lemma \ref{secondo lemma} reads:
\begin{equation} \label{intA}
    \iint_{[0,\pi]^2}  [L_{11}(\alpha_1,\alpha_2)+L_{22}(\alpha_1,\alpha_2)]L_{12}(\alpha_1,\alpha_2) \, d\alpha_1d\alpha_2 = - A(D)\int_0^\pi (p''(\alpha)+p(\alpha))^2 \,d\alpha.
    \end{equation}
\begin{lemma}
    For $L(\alpha_1,\alpha_2)=\omega(\gamma(\alpha_1),\gamma(\alpha_2))$ the next equality holds: 
\begin{multline}     \label{intB} 
2\iint_{[0,\pi]^2}L_{12}^2(\alpha_1,\alpha_2) \, d\alpha_1d\alpha_2  \\
        =\left(\int_0^\pi (p''(\alpha)+p(\alpha))^2 \, d\alpha\right)^2-\left(\int_0^\pi (p''(\alpha)+p(\alpha))^2\cos(2\alpha) \, d\alpha\right)^2-\left(\int_0^\pi(p''(\alpha)+p(\alpha))^2\sin(2\alpha) \, d\alpha\right)^2.
\end{multline}
\end{lemma}
\begin{proof} Same proof of \cite[Lemma 4.3]{BaBe}. We first observe that 
$$L_{12}^2(\alpha_1,\alpha_2)=(p''(\alpha_1)+p(\alpha_1))^2(p''(\alpha_2)+p(\alpha_2))^2\sin^2(\alpha_1-\alpha_2).$$
Since 
$$\sin^2(\alpha_1-\alpha_2)=\frac{1}{2}\left(1-\cos(2\alpha_1)\cos(2\alpha_2)-\sin(2\alpha_1)\sin(2\alpha_2)\right),$$ 
we take the double integral and conclude.
\end{proof}
\section{Proof of Theorem \ref{4-per}} \label{fine}
Under the hypotheses of Theorem \ref{4-per}, and by Proposition \ref{necessaria}, one of the next integral inequalities necessarily holds:
$$\iint_{\Omega_{\gamma\delta}}  [L_{11}(\alpha_1,\alpha_2)+2L_{12}(\alpha_1,\alpha_2)+L_{22}(\alpha_1,\alpha_2)]L_{12}(\alpha_1,\alpha_2) \, d\alpha_1d\alpha_2 \le 0$$
or 
$$\iint_{\Omega_{\delta \gamma^*}}  [L_{11}(\alpha_1,\alpha_2)+2L_{12}(\alpha_1,\alpha_2)+L_{22}(\alpha_1,\alpha_2)]L_{12}(\alpha_1,\alpha_2) \, d\alpha_1d\alpha_2 \le 0.$$
From Lemma \ref{primo lemma}, equalities (\ref{intA}) and (\ref{intB}) and the $\pi$-periodicity of $p$ and $p''$, the above inequalities read
\[
\begin{split}
&4\iint_{\Omega_{\gamma\delta}}  [L_{11}(\alpha_1,\alpha_2)+2L_{12}(\alpha_1,\alpha_2)+L_{22}(\alpha_1,\alpha_2)]L_{12}(\alpha_1,\alpha_2) \, d\alpha_1d\alpha_2 \\
&= 4\iint_{\Omega_{\delta \gamma^*}}  [L_{11}(\alpha_1,\alpha_2)+2L_{12}(\alpha_1,\alpha_2)+L_{22}(\alpha_1,\alpha_2)]L_{12}(\alpha_1,\alpha_2) \, d\alpha_1d\alpha_2 \\
&=-2A(D)\int_0^{2\pi} (p''(\alpha)+p(\alpha))^2 \,d\alpha 
    +\left(\int_0^{2\pi} (p''(\alpha)+p(\alpha))^2 \, d\alpha\right)^2 \\
&-\left(\int_0^{2\pi} (p''(\alpha)+p(\alpha))^2\cos(2\alpha) \, d\alpha\right)^2
-\left(\int_0^{2\pi}(p''(\alpha)+p(\alpha))^2\sin(2\alpha) \, d\alpha\right)^2 \le 0.
\end{split} 
\]
This is the same inequality of \cite[Beginning of Section 5]{BaBe}. We sketch the final argument for reader's convenience, see \cite[Section 5]{BaBe} for all details. \\
First, by the affine equivariance of the symplectic billiard map, we can apply an affine transformation $\varphi_{a,\sigma}$ defined as the composition of the unitary, diagonal, linear transformation $$\varphi_a(x,y)=\left( ax,\frac{y}{a}\right)$$ with $a>0$, after the rotation of angle $\sigma\in [0,2\pi]$. Let $D_{a,\sigma}:=\varphi_{a,\sigma}(D)$ and denote by $p_{a,\sigma}(\psi)$ its corresponding support function. By \cite[Proposition 5.1]{BaBe}, there exists a couple $(a,\sigma)\in \mathbb{R}_+\times[0,\frac{\pi}{2}]$ such that
\begin{equation*}
    \int^{2\pi}_0 p_{a,\sigma}(\psi)\cos(2\psi)\,d\psi=\int^{2\pi}_0 p_{a,\sigma}(\psi)\sin(2\psi)\,d\psi=0.
\end{equation*}
Secondly, we use the usual arc length as parameter to describe $D_{a,\sigma}$, see \cite[Remark 5.3]{BaBe}. \\
With these specific choices of $(i)$ affine transformation and $(ii)$ parametrization, the above necessary integral inequality simply reads:
$$\left(\int_0^{2\pi}(p''_{a,\sigma}+p_{a,\sigma})d\alpha\right)^2\leq 4\pi A(D_{a,\sigma}) \Longleftrightarrow (l(\partial D_{a,\sigma}))^2\leq 4\pi A(D_{a,\sigma}),$$ where $l(\partial D_{a,\sigma})$ is the perimeter length of $\partial D_{a,\sigma}$. By the planar isoperimetric inequality, this means that the domain $D_{a,\sigma}$ is necessarily a circle  and --from the very construction of $\varphi_{a,\sigma}$-- that the initial domain $D$ is an ellipse. \hfill $\Box$

\end{document}